\documentclass[journal]{IEEEtran}
\usepackage{amsmath,amsfonts,amssymb,amsthm}
\usepackage[caption=false,font=normalsize,labelfont=sf,textfont=sf]{subfig}
\usepackage{textcomp}
\usepackage{url}
\usepackage{graphicx}
\usepackage{cite}
\usepackage{hyperref}

\theoremstyle{remark}

\newtheorem{lemma}{Lemma}
\newtheorem{remark}{Remark}
\newtheorem{definition}{Definition}

\graphicspath{{figures}}

\begin{document}
\title{pqSEDMD: Subspace Methods for Extended Dynamic Mode Decomposition Identification}
\author{Camilo {Garc\'ia~Tenorio}, and Alain {Vande~Wouwer}}
\maketitle

\begin{abstract}
	Decomposition based methods, like the dynamic mode decomposition (DMD), along with its variant for nonlinear systems, the extended DMD (EDMD), are powerful tools for the extraction of meaningful spatio-temporal characteristics of (non)linear dynamical systems from measurement data. Despite some efforts to handle the identification task when dealing with real-world data, the decomposition based methods face a critical challenge: real-world data has two inherent sources of uncertainty, the process and measurement noise. Subspace identification methods, are robust tools able to provide accurate state-space models for multi-variable linear systems directly from input-output data. Combining these two methods, we introduce the p-q quasi-norm Subspace Extended Dynamic Mode Decomposition (pqSEDMD). An algorithm that uses our previous improvements to the EDMD by the use of a p-q-quasi-norm reduction on an orthogonal polynomial basis, the pqEDMD algorithm, along with subspace identification methods. The result is a robust approximation of nonlinear systems in a linear function space, combining the strengths of the two methodologies. Throughout the paper we will use the Duffing oscillator as a benchmark problem to show the effectiveness of the algorithm and illustrate many important aspects related to the development.
\end{abstract}

\section{Introduction} 
\label{sec:introduction}
The Dynamic Mode Decomposition (DMD)~\cite{Schmid2010} provides a data-driven framework for the analysis and prediction of linear dynamical systems and even some types of nonlinear flows~\cite{Rowley2009}. Despite its limitations to capture nonlinear behavior, it has found widespread applications for analysis~\cite{H_Tu_2014, Towne_2018, Bohon_2020} and control~\cite{Proctor2016}. To directly address the linear limitation of DMD, the extended DMD (EDMD)~\cite{Williams2015} introduced an approximation of nonlinear dynamics in a linear function space of observables, the so-called, Koopman operator~\cite{Koopman1931}. Korda and Mezi\'c provide the convergence guarantees of the EDMD to the Koopman operator at the limit of data and observables~\cite{Korda2018}. This convergence led to the proliferation of EDMD as the data-driven method to approximate the theoretically robust Koopman operator for the analysis of nonlinear dynamics. For example, Mauroy \textit{et al.}\cite{Mauroy_2020a} provide an overview of the research at the intersection of the Koopman operator theory and control theory.

The Koopman framework has since attracted substantial research interest: Klus \textit{et al.}\cite{Klus2018} extended the framework to stochastic systems; work related to the connection with physics-informed and manifold learning methods \cite{Baddoo2023,Otto2021}. The SINDy framework by Brunton \textit{et al.}~\cite{Brunton2016sindy} for discovering governing equations. For the subject of non-autonomous system identification, Haseli \textit{et al.}~\cite{Haseli_2026} provides a universal form showing the equivalence between considering the operator dynamics as a linear parameter varying system parameterized by the input~\cite{Williams_2016}, and as a bilinear form~\cite{Proctor2016, Korda2018a}. Despite recent efforts to approximate nonlinear dynamics and the Koopman operator for uncertain systems using maximum likelihood methods~\cite{Garcia-Tenorio2022}, the EDMD algorithm still faces numerical challenges that limit its implementation in real-world systems. Mainly, when there is uncertainty in the system dynamics and measurement noise.

Subspace identification methods~\cite{overschee96,verhaegen_verdult_2007} provide a robust framework for the identification of uncertain, linear, and time-invariant systems. These methods construct low-dimensional subspaces capturing dominant system modes through numerically stable algorithms based on singular value decomposition (SVD) and QR factorization. Their success in linear system identification raises a interesting question:
\textit{Is it possible to combine the robustness of subspace methods with the EDMD's ability to capture nonlinear dynamics in a linear function space?}

De Persis and Tesi~\cite{Persis2020} formulate a similar question when dealing with data-driven control methods. Also, recent work by~\cite{Shang_2024} use the Willems \textit{et al.}~\cite{Willems04} fundamental lemma to propose guarantees on the available data for identification under the assumption of the Koopman operator. Finally, van Waarde \textit{et al.}~\cite{van_Waarde_2020} also provide an interesting extension to the linear data-driven formulation when there are multiple data-trajectories available for the approximation of the subspace dynamics. This paper addresses the former question by introducing the p-q quasi-norm Subspace Extended Dynamic Mode Decomposition (pqSEDMD). This method explores the use of subspace methods to identify nonlinear dynamics using the EDMD algorithm with a p-q-quasi norm reduction of an orthogonal polynomial basis, referred to as pqEDMD~\cite{Garcia-Tenorio2022a}. The benefits of combining the two methods are: increasing the set of nonlinear dynamical system that can be identified just with data, and increasing the numerical robustness of EDMD methods and their approximations of the Koopman operator\footnote{The Matlab and Python implementations are publicly available in GitHub at \href{https://github.com/garten-cam/pqEDMDm}{pqEDMDm} and \href{https://github.com/garten-cam/pqEDMDp}{pqEDMDp} respectively. For the two packages, the class that implemets the pqSEDMD functionality is the ``sidDecomposition''.}.

\subsection{Motivating Example: The Duffing Oscillator}
\label{sub:Motivating Example: The Duffing Oscillator} 
The Duffing equation is the canonical benchmark to test the accuracy of algorithms and methods related to the EDMD algorithm and the Koopman operator. It is a nonlinear second order differential equation that can describe many types of dynamic behavior in nonlinear systems~\cite{Salas2021}. Depending on the parametrization, it's differential equation can describe a chaotic system, limit cycles or nonlinear mass-spring-damper systems:
\begin{equation}
	\ddot{x}+\delta\dot{x}+\alpha x + \beta x^3 = \gamma\cos(\omega t).
	\label{eq:DuffingEquation}
\end{equation}
Figure~\ref{fig:motivation} shows the performance of the pqSEDMD algorithm on the Duffing oscillator~\ref{eq:DuffingEquation} as a hardening spring system corrupted by noise.  The example uses a set of four trajectories to perform the identification, and two trajectories for testing the accuracy of the algorithm. Despite the challenges: non-linearity, non-uniqueness of the equilibrium, and measurement noise, the algorithm accurately reconstructs the system dynamics. Revealing the key advantages of out method: Robust identification of nonlinear dynamics in a linear function space for long term prediction and analysis.
\begin{figure}[ht]
	\begin{center}
		\includegraphics[width=0.95\linewidth]{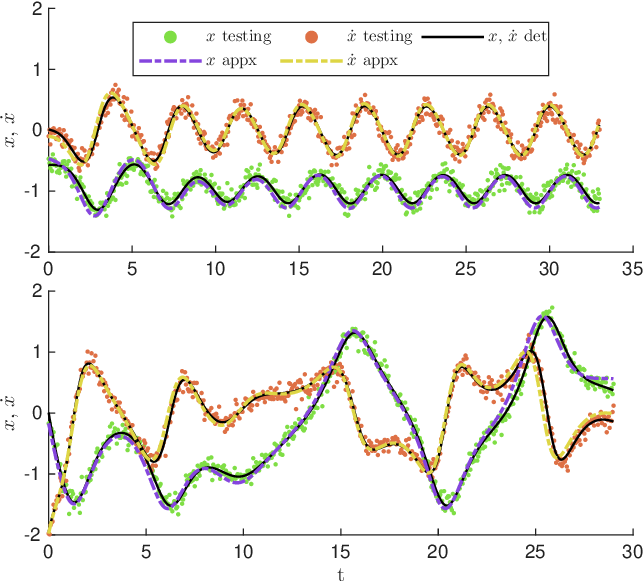}
	\end{center}
	\caption{Performance of the pqSEDMD algorithm on two testing trajectories of the Duffing oscillator where $\delta=0.5$, $\alpha=-1$, $\beta=1$, $\gamma\sim\texttt{U}(0,\,1)$, $\omega\sim\texttt{U}(-2,\,2)$, and with additive measurement noise $v\sim\texttt{N}(0,\,0.1)$.}
	\label{fig:motivation}
\end{figure}
The remaining of this paper explains how these results are possible. Section~\ref{sec:Data-drivenIdentificationMethods} introduces the necessary prerequisites, these are: the subspace identification methods in~\ref{sub:SubspaceIdentificationMethods}, the pqEDMD algorithm in~\ref{sub:pqEDMD Identification}, and the concepts from data driven control that support our claims. Section~\ref{sec:the pqsedmd} has our main result: the subspace identification algorithm in the context of the function space of the pqEDMD along with some numerical results. And finally, section~\ref{sec:Conclusions} presents some conclusions.
%
\section{Data-driven Methods}%
\label{sec:Data-drivenIdentificationMethods} 
Consider an input affine nonlinear dynamical system in discrete time
\begin{subequations}
	\label{eq:nonlin_affine_system}
	\begin{align}
		x(k+1) & = f(x(k)) + g(x(k)) u(k) + w(k) \label{eq:nonlin_affine_dyn} \\
		y(k)   & = h(x(k)) + v(k), \label{eq:nonlin_affine_out}
	\end{align}
\end{subequations}
where $x \in \mathbb{R}^n$ is the state, $u \in \mathbb{R}^m$ is the input, and $y \in \mathbb{R}^l$ is the output (i.e., the available measurements of the system). The system evolves according to the nonlinear state transition mapping $f\colon\mathbb{R}^n\rightarrow\mathbb{R}^n$, an input-to-state transition mappig $g\colon\mathbb{R}^m\rightarrow\mathbb{R}^n$, and the process noise $w(k) \sim \mathcal{N}(0, Q)$, a zero-mean white noise sequence with covariance matrix $Q \in \mathbb{R}^{n \times n}$. The output depends on the state via the state-to-output mapping $h\colon\mathbb{R}^n\rightarrow\mathbb{R}^{l}$, and is corrupted by the measurement noise $v(k) \sim \mathcal{N}(0, R)$, a zero-mean white noise sequence with covariance $R \in \mathbb{R}^{l \times l}$. The expectation of the process and output noise,
\[
	\mathbb{E}\left[
		\begin{pmatrix}
			w(i) \\v(i)
		\end{pmatrix}
		\begin{pmatrix}
			w^{T}(j) & v^{T}(j)
	\end{pmatrix}\right] =
	\begin{pmatrix}
		Q & S \\ S^T & R
	\end{pmatrix} \delta_{ij} \geq 0.
\]

The solution of~\eqref{eq:nonlin_affine_system} produces the sequences of states $x=\{x_k\}_{k=0}^{T-1}\in\mathbb{R}^{n\times T}$ and outputs $y=\{y_k\}_{k=0}^{T-1}\in\mathbb{R}^{l\times T}$ that come from the successive application of the state transition mapping from an initial condition $x(0)$ and the application of a known sequence of inputs $u=\{u_k\}_{k=0}^{T-1}\in\mathbb{R}^{m\times T}$ according to the input and output functions. Then, from a trajectory of length $T$, define the data matrices
\begin{subequations}
	\label{eq:dataMat}
	\begin{align}
		Y & =
		\begin{bmatrix}y(0) & \ldots & y(T-1)
		\end{bmatrix}\label{eq:dataY} \\
		U & =
		\begin{bmatrix}u(0) & \ldots & u(T-1)
		\end{bmatrix}.\label{eq:dataU}
	\end{align}
\end{subequations}

For the descriptions of the subspace methods, we assume that there is a trajectory of the system of sufficient length and information, available for the identification. Conversely, for the pqEDMD, and the pqSEDMD algorithms, we assume that there are many of these trajectories available for \textit{training} the algorithms, and some others for \textit{testing} their accuracy.
\subsection{Subspace Identification}\label{sub:SubspaceIdentificationMethods} 
The objective of a subspace identification algorithm is to find a linear approximation of a controllable and observable dynamical system from input and output data~\eqref{eq:dataMat}. These algorithms estimate linear state space models, i.e., linearized versions of~\eqref{eq:nonlin_affine_system} that take the form:
\begin{subequations}
	\label{eq:lin_sys}
	\begin{align}
		x(k+1) & = Ax(k) + Bu(k) + w(k) \label{eq:discrete_system1} \\
		y(k)   & = Cx(k) + v(k) \label{eq:discrete_system2}
	\end{align}
\end{subequations}
where $A\in\mathbb{R}^{n\times n}$ is the state transition matrix that governs the evolution of the unknown state, $B\in\mathbb{R}^{m\times n}$ is the input-to-state transition matrix, and $C\in\mathbb{R}^{n\times l}$ is the output matrix. The linear approximation of the dynamics assumes that the input-output data lies in the vicinity of a hyperbolic fixed point of the nonlinear system~\eqref{eq:nonlin_affine_system}. The fixed point satisfies the Hartman-Grobman theorem~\cite{giesl2015review}. The approximation further assumes that the sequence of inputs is persistently exiting~\cite{Willems04,Persis2020}.

Ignoring the uncertainties of the system, i.e., assuming $w(k)=0$ and $v(k)=0$ in~\eqref{eq:lin_sys}. The solution, or an orbit of the system is the successive application of the state transition matrix to the state, along with the application of the sequence of forcing signals, starting from a specific initial condition $x(0)=x_0$,
\begin{equation}
	x(k)=A^{k} x(0) + \displaystyle\sum_{i=0}^{k-1}A^{k-i-1}Bu(i).
	\label{eq:DS_solution}
\end{equation}

Following the (standard) notation of Van Overshee and De Moore~\cite{overschee96}, define the input and output block Hankel matrices, $U_{0\vert2i-1} \in\mathbb{R}^{2mi\times j}$ and $Y_{0\vert2i-1} \in\mathbb{R}^{2li\times j}$, where the subscripts of the Hankel matrices indicate the initial time of the signal and the number of blocks, i.e., the number of stored samples of a trajectory in each column of the matrix, for example, the input Hankel matrix is,
{\footnotesize
	\begin{align}
		& U_{0\vert2i-1}  \triangleq \nonumber \\
		&
		\begin{bmatrix}
			u(0)    & u(1)   & u(2)    & \cdots & u(j-1)    \\
			u(1)    & u(2)   & u(3)    & \cdots & u(j)      \\
			\vdots  & \vdots & \vdots  & \ddots & \vdots    \\
			u(i-1)  & u(i)   & u(i+1)  & \cdots & u(i+j-2)  \\
			\hdotsfor{5}                                    \\
			u(i)    & u(i+1) & u(i+2)  & \cdots & u(i+j-1)  \\
			u(i+1)  & u(i+2) & u(i+3)  & \cdots & u(i+j)    \\
			\vdots  & \vdots & \vdots  & \ddots & \vdots    \\
			u(2i-1) & u(2i)  & u(2i+1) & \cdots & u(2i+j-2)
		\end{bmatrix}\nonumber
\end{align}}
\begin{equation}
	=
	\begin{bmatrix}
		U_{0|i-1} \\\hdotsfor{1}\\U_{i|2i-1}
	\end{bmatrix} =                      %
	\begin{bmatrix}
		U_{p} \\\hdotsfor{1}\\U_{f}
	\end{bmatrix},
	\label{eq:input_Hankel}
\end{equation}
where $U_p$ denotes the Hankel matrix of \textit{past} values and $U_f$ denotes the Hankel matrix of \textit{future} values, where each matrix has $i$ Hankel blocks. This division is related to the use of instrumental variables that will handle the uncertainties in the approximation.

The output Hankel matrix has the same structure as the input matrix~\eqref{eq:input_Hankel} along with the same division of \textit{past} and \textit{future} matrices with $i$ blocks each,
\begin{equation}
	Y_{0|2i-1} = %
	\begin{bmatrix}
		Y_{0|i-1} \\\hdotsfor{1}\\Y_{i|2i-1}
	\end{bmatrix}=
	\begin{bmatrix}
		Y_{p} \\\hdotsfor{1}\\Y_{f}
	\end{bmatrix}.
	\label{eq:output_Hankel}
\end{equation}
Furthermore, consider the shifted matrices with the $\pm$ superscript, these matrices include or exclude a block from the respective \textit{past} and \textit{future} matrices, i.e., $U_p^+=U_{0|i}$ $Y_p^+=Y_{0|i}$ (includes an additional block), and $U_f^-=U_{i+1|2i-1}$ and $Y_f^-=Y_{i+1|2i-1}$ (excludes the first block). Whatever the case, the number of columns is the same, and it is related to the number of available data points $T$ and the quantity of of Hankel blocks $i$. Therefore, the number of columns is $j=T - 2i + 1$. This implies that for a particular selection of Hankel blocks, the number of samples $T>2i + j - 1$.

Considering the solution of the system~\eqref{eq:DS_solution}, and the output equation~\eqref{eq:discrete_system2}, the input-output relationship of the Hankel matrices is the following data equation,
\begin{equation}
	Y_{0|2i-1}  = \Gamma_{2i}X_{0\colon T-2i}+ H_{2i} U_{0|2i-1},
	\label{eq:system_obs_mark}
\end{equation}
where $X_{0\colon T-2i}$ is an unknown sequence of states, $\Gamma_{2i}\in\mathbb{R}^{2il\times n}$ is an extended observability matrix,
\begin{equation}
	\Gamma_{2i}=
	\begin{bmatrix}
		C \\CA\\CA^2\\\vdots\\CA^{2i-1}
	\end{bmatrix}
	\label{eq:Extended_observability_mat}
\end{equation}
and $H_{2i}\in\mathbb{R}^{2il\times 2im}$ is a lower block triangular Toeplitz matrix, often-called, the Markov parameters of the system,
\begin{equation}
	H_{2i}=
	\begin{bmatrix}
		D          & 0          & \cdots & 0      & 0      \\
		CB         & D          & \cdots & 0      & 0      \\
		CAB        & CB         & \cdots & 0      & 0      \\
		\vdots     & \vdots     & \ddots & \vdots & \vdots \\
		CA^{2i-2}B & CA^{2i-3}B & \cdots & CB     & D
	\end{bmatrix}.
	\label{eq:markov_parameter}
\end{equation}

Starting from~\eqref{eq:system_obs_mark} and using the \textit{past} and \textit{future} division of the Hankel matrices, the identification problem becomes,
\begin{subequations}
	\label{eq:input_output_matrix}
	\begin{align}
		Y_p & = \Gamma_i X_p + H_i U_p \label{eq:input_output_matrix_p}  \\
		Y_f & = \Gamma_i X_f + H_i U_f \label{eq:input_output_matrix_f},
	\end{align}
\end{subequations}
with the appropriate dimension of the observability matrix and Markov parameters, and the division of the state sequence,
{\small
	\begin{subequations}
		\label{eq:state_sequence_pf}
		\begin{align}
			X_p & =                                     %
			\begin{bmatrix}
				x(0) & x(1) & \cdots & x(T-2i)
			\end{bmatrix}\label{eq:state_sequence_past} \\
			X_f & =                                     %
			\begin{bmatrix}
				x(i) & x(i + 1) & \cdots & x(T-i)
			\end{bmatrix},\label{eq:state_sequence_future}
		\end{align}
\end{subequations}}
either~\eqref{eq:input_output_matrix_p} or~\eqref{eq:input_output_matrix_f} have the following matrix representation,
\begin{equation}
	\begin{bmatrix}
		Y_{\{p,f\}} \\ U_{\{p,f\}}
	\end{bmatrix} =%
	\begin{bmatrix}
		\Gamma_i & H_i \\ 0 & I_{im}
	\end{bmatrix}%
	\begin{bmatrix}
		X_{\{p,f\}} \\ U_{\{p,f\}}
	\end{bmatrix},
	\label{eq:io_matrix_form}
\end{equation}
where $I_{im}$ is an identity matrix. The reason to have a distinction between \textit{past} and \textit{future}, is to deal with the uncertainties in the system: the measurement and the process noise. In this context, define the instrumental variables of the algorithm as,

\begin{subequations}
	\label{eq:instrumental_variables}
	\begin{align}
		W_{0|i-1} & \triangleq
		\begin{bmatrix}
			U_{0|i-1} \\Y_{0|i-1}
		\end{bmatrix} \label{eq:instrumental_1} \\
		& =
		\begin{bmatrix}
			U_p \\Y_p
		\end{bmatrix} \label{eq:instrumental_2} \\
		& = W_p,
	\end{align}
\end{subequations}
and the same for $W_p^{+}$ that represents the same grouping of the inputs and outputs for the past time-frames, shifted over one time-instant.

The methods we are using regarding subspace identification rely on the ability to estimate a sequence of states from the input-output data, and subsequently, estimate the system matrices.

Starting with the deterministic case, consider equation~\eqref{eq:system_obs_mark}, post-multiplied by a matrix $\Pi_{U^{\bot}}$ that satisfies $U_{0|2i-1} \Pi_{U^{\bot}}=0$. The result is,
\begin{equation}
	Y_{0|2i-1}\Pi_{U^\bot} = \Gamma_{2i} X_{0\colon T-2i}\Pi_{U^\bot},
	\label{eq:column_space}
\end{equation}
where,
\begin{equation}
	\Pi_{U^\bot} = I - U_{0|2i-1}^T(U_{0|2i-1}U_{0|2i-1}^T)^{-1}U_{0|2i-1},
	\label{eq:pi_u_orthogonal}
\end{equation}
and is equivalent to: having a geometric operator that projects the row space of the Hankel matrix of outputs onto the orthogonal complement of the row space of the Hankel matrix of inputs. For the deterministic case, define $\mathcal{O}_{2i}$ as the projection:
\begin{equation}
	\mathcal{O}_{2i} = Y_{0|2i-1}\Pi_{U^\bot} =Y_{0|2i-1}\Big/\mathbf{U}_{0|2i-1}^\bot,
	\label{eq:Oi_mat}
\end{equation}
where the $A\big/ \mathbf{B}^{\bot}$ operator denotes the orthogonal projection of $A$ onto the orthogonal complement of $B$.

For the forced-stochastic case, the method describes a similar calculation that uses the instrumental variables of the system~\cite{overschee96,verhaegen_verdult_2007}. Then, $\mathcal{O}_i$ is the oblique projection of the future outputs onto the row space of the instrumental variables, along the row space of future inputs:
\begin{equation}
	\mathcal{O}_i =  Y_f\Big/_{U_f}\mathbf{W}_p,
	\label{eq:Oi_stochastic}
\end{equation}
where the $A\big/_{C}\mathbf{B}$ operator denotes the oblique projection of $A$ along the row space of $C$, onto the row space of $B$.

The last choice before proceeding with the estimation of the state sequence and the system matrix approximation is the selection of subspace algorithm. There are different subspace algorithms that use different weighting matrices for $\mathcal{O}_i$. These matrices are: the full rank matrix $W_1\in\mathbb{R}^{li\times li}$ and $W_2\in\mathbb{R}^{j\times j}$ that satisfies $\texttt{rank}(W_p)=\texttt{rank}(W_pW_2)$~\cite{overschee96}. For example, the \textit{multivariable output error state space} (MOESP) algorithm uses $W_1=I_{li}$ and $W_2=\Pi_{U_f^{\bot}}$. Notice, that similarly to~\eqref{eq:Oi_mat} post multiplying $O_i$ by $\Pi_{U_f^{\bot}}$ is equivalent to having a geometric operator that projects the row space of $O_i$ onto the orthogonal complement of the row space of the Hankel matrix of future inputs,
\begin{equation}
	W_1O_iW_2 = I_{li}O_i\Pi_{U_f^{\bot}} = O_i\Big/\mathbf{U}_f^\bot.
	\label{eq:WOiW_projection}
\end{equation}

With the definition of $W_1\mathcal{O}_iW_2$, the approximation of the extended observability matrix comes from the singular value decomposition,
\begin{align}
	W_1\mathcal{O}_iW_2 & =                                                            %
	\begin{bmatrix}
		\mathcal{U}_1 & \mathcal{U}_2
	\end{bmatrix}%
	\begin{bmatrix}
		\Sigma_1 & 0 \\
		0        & 0
	\end{bmatrix}%
	\begin{bmatrix}
		\mathcal{V}_1^T \\ \mathcal{V}_2^T
	\end{bmatrix}\nonumber                                                  \\
	& = \mathcal{U}_1\Sigma_1\mathcal{V}_1^T,\label{eq:linear_svd}
\end{align}
where the $1$ subscript denotes the truncation of the matrices according to the nonzero singular values. Then, an approximation of the extended observability matrix is
\begin{equation}
	\hat{\Gamma}_i = W_1^{-1}\mathcal{U}_1\Sigma_1^{1/2},
	\label{eq:Ext_obs_mat}
\end{equation}
and a sequence of states that explains the dynamics of the system is,
\begin{equation}
	\hat{X}_{i} = \Gamma_i^\dagger\mathcal{O}_i,
	\label{eq:state_seq}
\end{equation}
where $^{\dagger}$ denotes the Moore-Penrose pseudo inverse. Notice that these approximations of the extended observability matrix and the approximation of the state sequence are not equivalent to their definitions in~\eqref{eq:Extended_observability_mat} and~\eqref{eq:state_sequence_pf}, i.e., $\hat{\Gamma}_i\ne\Gamma_i$ and $X_i\ne \hat{X}_i$. These relationships become an equality under the action of a similarity transformation. Given that the objective is to have some extended observability matrix, consistent with some sequence of states, it is not necessary to derive the similarity transformation.

The final element, necessary for the approximation of the system matrices, is  a shifted sequence of states, the result of evolving $\hat{X}_i$ one step ahead in time, and solve a least squares problem to get the system matrices $A$, $B$, $C$, and $D$. Then, define the shifted set of future states as $\hat{X}_{i+1}$, and  $\mathcal{O}_{i-1}$ from the oblique projection
\begin{equation}
	\mathcal{O}_{i-1} = Y_f^-/_{U_f^-}(W_p^+).
	\label{eq:shift_oblique_projection}
\end{equation}
Also, define the truncated observability matrix $\Gamma_{i-1}$ as the observability matrix without the last block, i.e., $\hat{\Gamma}_i$ without the last $l$ rows. Then, the sequence of shifted future states is,
\begin{equation}
	\hat{X}_{i+1} = \Gamma_{i-1}^{\dagger}\mathcal{O}_{i-1}.
	\label{eq:shifted_future_states}
\end{equation}
With the two sequences of states, the first Hankel block of the \textit{future} outputs, and the \textit{future} inputs, the solution of the following least squares problem yields the matrices of the system,
\begin{equation}
	\begin{bmatrix}
		\hat{X}_{i+1} \\ Y_{i|i}
	\end{bmatrix} =
	\begin{bmatrix}
		A & B \\ C & D
	\end{bmatrix}
	\begin{bmatrix}
		\hat{X}_{i} \\U_{i|i}
	\end{bmatrix}.
	\label{eq:matrices_least_squares}
\end{equation}

This is a robust and well tested algorithm that can identify the linear dynamics of a system, with the clear disadvantage of not being able to capture the nonlinear dynamics of a system.
\subsubsection{Rank Conditions and Persistence of Excitation}\label{sec:Rank Conditions and Persistence of Excitation} 
For the special case of deterministic subspace identification of autonomous systems, the data equation~\eqref{eq:system_obs_mark} reduces to,
\begin{equation}
	Y_{0|2i-1} = \Gamma_{2i}X_{0\colon T-2i}.
	\label{eq:autonomous_system_obs}
\end{equation}
Without inputs, $Y_{0|2i-1}$ is a linear combination of the columns of the extended observability matrix. Implying that the column space of the observability matrix contains the column space of the Hankel matrix of outputs $\texttt{col}(Y_{0|2i-1})\subseteq\texttt{col}(\Gamma_{2i})$; it does not imply that the rank of the matrices is equal, because it is possible to have that the $\texttt{rank}(Y_{0|2i-1})<\texttt{rank}(\Gamma_{2i})$. To conclude that in fact there is an equality, the following conditions must be met: (1) The number of of Hankel blocks is greater than the order of the system, i.e., $2i>n$. (2) The number of columns in the Hankel matrix of outputs is greater than the number of Hankel blocks, i.e., $j>2i$. Finally (3) the sequence of states $X_{0\colon T-2i}$ has full row rank $n$, i.e., $\texttt{rank}(X_{0\colon T-2i})=n$. The first two conditions depend on the available number of data points and the choice of Hankel blocks. The third condition depends on the initial condition and the length of the state sequence.
\begin{lemma}[dentification for autonomous systems \cite{verhaegen_verdult_2007}, Section 9.2.2.]
	Consider an autonomous system~\eqref{eq:lin_sys}. If the system is observable, if the initial condition $x(0)=x_0$ is informative enough, and the resulting sequence of outputs is long enough to have $2i>n$ Hankel blocks and $j\geq n$ columns, then,
	\begin{equation}
		\texttt{rank}(Y_{0|2i-1})=n
		\label{eq:hankel_out_rank_cond}
	\end{equation}
\end{lemma}
\begin{proof}
	From the assumption that the system is observable, $\texttt{rank}(\Gamma_{2i})=n$. The $T-2i$ sequence of states from a non-zero initial condition is,
	and under the assumption that the trajectory from the initial condition is informative enough, $\texttt{rank}(X_{0\colon T-2i})=n$. Applying Sylvester's inequality to~\eqref{eq:autonomous_system_obs}, implies that the $\texttt{rank}(Y_{0|2i-1})=n$.
\end{proof}
\begin{remark}
	This condition on the Hankel matrix of outputs implies the identifiability of the $(A,\,C)$ pair. Meaning that a sufficiently informative trajectory in the absence of an input, is still sufficient for the identification of the $A$ and $C$ matrices. The rank condition on the sequence of inputs from the fundamental lemma is sufficient and necessary only under zero initial conditions of a non-autonomous system.
\end{remark}

In summary, if there are enough points in the dataset and those points come from a state sequence whose initial condition provides enough information, then the sequence of states will have full row $\texttt{rank}(X_{0\colon T-2i})=n$, making the extended observability matrix identifiable.

For the general case of non autonomous systems, the possibility to identify the dynamics comes from a rank condition of the sequence of states, the Hankel matrix of inputs, and the Hankel matrix of outputs. The rank conditions that these matrices must satisfy is the so-called persistence of excitation lemma from Willems~\cite{Willems04}.

\begin{definition}[Persistently exciting sequence]
	\label{def:persistence_excitation}
	The sequence $s=\{s_k\}_{k=0}^{T}\in\mathbb{R}^{\sigma\times T}$ is persistently exciting of order $i$ if the associated Hankel matrix $S_{0|i-1}$ has full row rank, i.e., $\texttt{rank}(S_{0|i-1})=\sigma i$.
\end{definition}

For a state sequence $X_{0\colon T-2i}$ and the corresponding Hankel matrix of inputs $U_{0|2i-1}$, the rank condition to satisfy is:
\begin{lemma}[From~\cite{Willems04} Cor. 2, and~\cite{verhaegen_verdult_2007}  Lem. 9.1]
	\label{lem:state_input_sequence}
	Consider a controllable system~\eqref{eq:lin_sys}, and an input-state trajectory $(x,u)=\{x_k;u_k\}_{k=0}^{T-1}\in\mathbb{R}^{(n+m)\times T}$ from zero initial conditions. If $u$ is persistently exciting of order $n+2i$, then
	\[\texttt{rank}\left(
			\begin{bmatrix}X_{0\colon T-2i} \\U_{0|2i-1}
	\end{bmatrix}\right)=n+2im,\]
	then,
	\[\texttt{rank}(Y_{0|2i-1}\Pi_{U^{\bot}})=n\]
	and,
	\[\texttt{col}(Y_{0|2i-1}\Pi_{U^{\bot}})=\texttt{col}(\Gamma_{2i}),\]
\end{lemma}
Lemma~\ref{lem:state_input_sequence} proves that the observability matrix $\Gamma_{n}$ of a system is identifiable, and the remaining process to get the system matrices feasible.
Finally, the fundamental lemma is,
\begin{lemma}[Fundamental lemma~\cite{Markovsky_2005} Lem. 1]
	Consider a controllable system~\eqref{eq:lin_sys}, and an input to output trajectory $(u_{0|T-1},y_{0|T-1})=\{u_{0|T-1};\}\in\mathbb{R}^{(m+l)\times T}$. If the input sequence $u_{0|T-1}$ is is persistently exciting of order $n+2im$. Then,
	\begin{enumerate}[label=(\roman*)]
		\item Any $2i$-length trajectory $(u_{0|2i-1},y_{0|2i-1})$ of the linear system~\eqref{eq:lin_sys} can be written as a linear combination of of the columns of the combined Hankel matrix of inputs and outputs,
			\begin{equation}
				\begin{bmatrix}U_{0|2i-1} \\Y_{0|2i-1}
				\end{bmatrix}
				\label{eq:hankel_matrix_input_output}
			\end{equation}
		\item Any linear combination of~\eqref{eq:hankel_matrix_input_output} is a $2i$ length trajectory of the system.
	\end{enumerate}
\end{lemma}
\subsection{pqEDMD Identification}\label{sub:pqEDMD Identification} 
The objective of the EDMD algorithm is to give an approximation in discrete-time of the time-evolution of a dynamical system~\eqref{eq:nonlin_affine_system} from input and output data~\eqref{eq:dataMat}, assuming that the full state is available at output, i.e., $H=I_{n}$. The data can come from the measurements of a real system or the numerical integration of a differential equation. The result is a nonlinear system of equations that describes the time-evolution of a set of nonlinear functions (the observables) $\Psi\colon\mathbb{R}^l\rightarrow\mathbb{C}^d$ by the effect of a matrix, i.e., a \textit{linear} operator. Then, nonlinear time-evolution of the original dynamical system~\eqref{eq:nonlin_affine_system} becomes the linear evolution of the function space $\Psi(y)$ under the effect of the operator,
\begin{subequations}
	\label{eq:EDMD}
	\begin{align}
		\Psi(y(k+1)) & = A_D\Psi(y(k)) + B_D u(k)\label{eq:EDMD1} \\
		y(k)         & = C_D \Psi(y(k)),\label{eq:EDMD2}
	\end{align}
\end{subequations}
where $A_D\in\mathbb{R}^{d\times d}$ is the operator on the function space that evolves the observables, $B_D\in\mathbb{R}^{d\times m}$ is the input to observable matrix, and $C_D\in\mathbb{R}^{l\times d}$ is the observables to output matrix.

For the approximation of the dynamics via the EDMD algorithm, assume that there is an available data-set of sufficient size and information, and that this data-set is not restricted to a single trajectory. Then, assume that there are $E$ available experiments, or sample trajectories, where each trajectory has $T_{i}$ data-points for $i=1,\,\dots,\,E$. For each one of  these experiments, divide the data into two distinct sequences, the \textit{past} and \textit{future} values ($_+$ and $_-$ respectively, instead of $_p$ and $_f$ to avoid confusion with the previous section). These data sets are:
\begin{subequations}
	\label{eq:edmd_data_mat_i}
	\begin{align}
		\left\{Y^{i}_+\right.             & = \left.
			\begin{bmatrix}y^{i}(1) & y^{i}(2) & \cdots & y^{i}(T_i)
		\end{bmatrix}\right\}_{i=1}^E   \\
		\left\{Y^{i}_-\right.             & = \left.
			\begin{bmatrix}y^{i}(0) & y^{i}(1) & \cdots & y^{i}(T_i-1)
		\end{bmatrix}\right\}_{i=1}^E \\
		\left\{U^{i}_{\phantom{+}}\right. & = \left.
			\begin{bmatrix}u^i(0) & u^i(1) & \cdots & u^i(T_i-1)
		\end{bmatrix}\right\}_{i=1}^E,
	\end{align}
\end{subequations}
notice that each $y_+^i$ is the value of the output one step ahead of $y_-^i$. From the complete set of data matrices~\eqref{eq:edmd_data_mat_i}, the data-matrix to calculate EDMD is the horizontal concatenation of all the \textit{future} and \textit{past} matrices,
\begin{subequations}
	\label{eq:edmd_data_mat}
	\begin{align}
		Y_+             & =
		\begin{bmatrix}Y^{1}_+ & Y^{2}_+ & \cdots & Y^{E}_+
		\end{bmatrix} \\
		Y_-             & =
		\begin{bmatrix}Y^{1}_- & Y^{2}_- & \cdots & Y^{E}_-
		\end{bmatrix} \\
		U_{\phantom{+}} & =
		\begin{bmatrix}U^{1}_{\phantom{+}} & U^{2}_{\phantom{+}} & \cdots & U^{E}_{\phantom{+}}
		\end{bmatrix},
	\end{align}
\end{subequations}
meaning that for $E$ trajectories, the total number of available points $N=\sum_{i=1}^{E}(T_i-1)$.

The set of observables comes from a family of orthogonal polynomials, where each individual element of the set comes from the product of $l$ univariate polynomials. Additionally, a vector of orders $\alpha\in\mathbb{N}_0^{l}$ defines each of the multivariate polynomials, and the selection criteria for the inclusion of an observable comes from the comparison of the q-quasi norm of the orders $\alpha$ against the maximum polynomial order $p$. The definition of the q-quasi norm is $\Vert\alpha\Vert=(\sum_{i=1}^l\alpha_i^q)^{1/q}$ for $q\in\mathbb{R}$. Then, if
\begin{equation}
	\left\{\alpha\in\mathbb{R}^l\colon\Vert\alpha\Vert\leq p\right\},
	\label{eq:alpha_inclusion}
\end{equation}
the polynomial does not get excluded from the basis~\cite{Garcia-Tenorio2022a}.

With the data matrices for the pqEDMD identification and the set of observables, the system matrices come from the solution of the following least squares problems,
\begin{subequations}
	\label{eq:pqedmd_least_squares}
	\begin{align}
		\Psi(Y_{+}) & =
		\begin{bmatrix}A_D & B_D
		\end{bmatrix}
		\begin{bmatrix}\Psi(Y_{-}) \\U
		\end{bmatrix} + r_{[A_D\,B_D]}\label{eq:AB_least_sq}           \\
		Y_{-}       & = C_D\Psi(Y_{-}) + r_{C_D},\label{eq:C_least_sq}
	\end{align}
\end{subequations}
where $r_{[A_D\,B_D]}$ and $r_{C_D}$ are the residual terms to minimize.

From the plethora of methods to solve least squares problems, we want emphasize the use of a singular value decomposition of the regression matrix in~\eqref{eq:AB_least_sq}, along with the determination of its effective rank to perform the regression. This process leads to an orthogonalization of the observables and from the authors experience, better approximations~\cite{Garcia-Tenorio2025}. For the second least squares problem~\eqref{eq:C_least_sq}, instead of a numerical solution, it is possible to exploit the structure of the order one polynomials in the set of observables to recover the state~\cite{Garcia-Tenorio2022a}.

Over the subspace methods, this algorithm has the advantage of providing the approximation of the nonlinear dynamics of a system. With the disadvantage of not being robust to uncertainties.
\section{The pqSEDMD}\label{sec:the pqsedmd} 
This section shows how the shortcomings of the two previous algorithms are alleviated via their combination. With the lifting of the data in the function space we can identify the nonlinear dynamics in a linear function space, where the subspace methods provide a robust identification method that deals with uncertainties and increases the accuracy of the prediction in the function space.

\subsection{Lifting and Hankel Matrices}\label{sub:Lifting and Hankel Matrices} 
As with the pqEDMD, we assume that there are an $E$ number of available experiments for the identification, where each trajectory has $T_i$ data-points. Instead of dividing the available experiments into the past and future matrices, the treatment for the pqSEDMD algorithm is to lift the output with the set of observables, and then apply the Hankel transformation. As with the pqEDMD, we will not apply any lifting to the input. Different from the pqEDMD, there will be no concatenation of the data matrices before the lifting or Hankel transformation, instead, we carry the individual matrices and concatenate whenever necessary.

In contrast with the traditional subspace methods, where the number of Hankel blocks is the same for the \textit{past} and \textit{future} Hankel matrices, this development considers these two quantities to be different. Empirical verification shows that having the same number of blocks for the Hankel matrices produces numerical instabilities for the algorithm. Therefore, the lifted Hankel matrices of outputs and the Hankel matrices of inputs will have $b_f$ future blocks and $b_p$ past blocks. The set of Hankel matrices of the lifted output is,
\[
	\left\{\Psi^{i}_{0|b_f+b_p-1}\in\mathbb{R}^{(b_f+ b_p)d \times j_i}\right\}_{i=1}^{E},
\]
where the $\Psi^{i}_{0|b_H-1}$ denotes the Hankel matrix with $b_H=b_p+b_f$ blocks of the lifted matrix $Y^i$, i.e., $\Psi(Y^i)$. Then, each element of the set is,
{\footnotesize
	\begin{align}
		& \Psi^{i}_{0|b_H-1}\triangleq \nonumber \\
		&
		\begin{bmatrix}
			\Psi(y(0))     & \Psi(y(1))     & \cdots & \Psi(y(j_i-1))     \\
			\Psi(y(1))     & \Psi(y(2))     & \cdots & \Psi(y(j_i))       \\
			\vdots         & \vdots         & \ddots & \vdots             \\
			\Psi(y(b_p-1)) & \Psi(y(b_p))   & \cdots & \Psi(u(b_p+j_i-2)) \\
			\hdotsfor{4}                                                  \\
			\Psi(y(b_p))   & \Psi(y(b_p+1)) & \cdots & \Psi(y(b_p+j_i-1)) \\
			\Psi(y(b_p+1)) & \Psi(y(b_p+2)) & \cdots & \Psi(y(b_p+j_i))   \\
			\vdots         & \vdots         & \ddots & \vdots             \\
			\Psi(y(b_H-1)) & \Psi(y(b_H))   & \cdots & \Psi(y(b_H+j_i-2))
		\end{bmatrix}\nonumber
\end{align}}
\begin{equation}
	=
	\begin{bmatrix}
		\Psi^{i}_{0|b_p-1} \\\hdotsfor{1}\\\Psi^{i}_{b_p|b_H-1}
	\end{bmatrix}=
	\begin{bmatrix}
		\Psi^{i}_p \\\hdotsfor{1}\\\Psi^{i}_f
	\end{bmatrix}.
	\label{eq:lifted_hankel_output}
\end{equation}

For the input set of signals from the different experiments, we apply a similar approach, and gather the different Hankel matrices similar to~\eqref{eq:input_Hankel}, where instead of
$2i$ Hankel blocks, there are $b_p+b_f$ in each matrix of the set,
\[
	{\left\{U^i_{0|b_p+b_f-1}\in\mathbb{R}^{(b_p+b_f)m\times j_i}\right\}}_{i=1}^{E}.
\]

As in the linear case, the reason to have a division between \textit{past} and \textit{future} blocks is to use the instrumental variables that will deal with the noise. Then, define the set of instrumental variables,
\[
	{\left\{W^{i}_{\Psi_p}\in\mathbb{R}^{b_p(d+m)\times j_i}\right\}}_{i=1}^{E},
\]
and the elements of this set are the matrices,
\begin{equation}
	W^{i}_{\Psi_p} = W^{i}_{0|b_p-1}\triangleq
	\begin{bmatrix}
		U^{i}_{0|b_p-1} \\\Psi^{i}_{0|b_p-1}
	\end{bmatrix}.
	\label{eq:lifted_Wp}
\end{equation}
For all the previous matrices~\eqref{eq:lifted_hankel_output} and~\eqref{eq:lifted_Wp}, the number of columns for each element of the sets is $j_i=T_i - (b_p+b_f) + 1$ for $i=1,\,\dots,\,E$.

For the evaluation of the signals (\textit{lifting}) in the functions space, and the calculation of the Hankel matrices, there are some parameters that affect the outcome of the overall algorithm. For the observables, there is: the family of orthogonal polynomials, where, depending on the noise distribution, and the applied normalization or lack of normalization, some families may perform better than others. Then, there are the p-q parameters for the pq-quasi norm reduction, that in general come from a sweep or a Montecarlo simulation. For the Hankel matrices, the two parameters are the number of blocks for the \textit{past} and the \textit{future}, their selection also depends on the noise distribution.

All of the aforementioned parameters, except for the family of polynomials, will affect the order of the subspace. The following sections will provide more insight related to the implications of the parameter selection.
\subsection{Lifted Subspace}\label{sub:Lifted Subspace} 
We formulate the lifted subspace identification problem as a linear time-invariant dynamical system in a function space $\Phi\colon\mathbb{R}^n\rightarrow\mathbb{R}^{d_s}$, where $\Phi(x)$ denotes a \textit{lifting} of the state $x\in\mathbb{R}^n$ into a higher-dimensional space, and we assume the existence of a linear operator $A_S\in\mathbb{R}^{d_s\times d_s}$ that governs the dynamics in the function space. The linear evolution of the lifted state and its relationship with the original set of observables and outputs is,
\begin{subequations}
	\label{eq:lifted_subspace}
	\begin{align}
		\Phi(x(k+1)) & = A_S\Phi(x(k)) + B_Su(k)\label{eq:ls_dynamics} \\
		\Psi(y(k))   & = C_S\Phi(x(k))\label{eq:ls_psi_output}         \\
		y(k)         & = C_D\Psi(y(k))\label{eq:ls_y_output}
	\end{align}
\end{subequations}
where, $B_S \in \mathbb{R}^{d_s \times m}$ is the input-to-lifted-state transition matrix, $C_S \in \mathbb{R}^{d \times d_s}$ maps the lifted states to the lifted outputs, and $C_D$ is the same output projection matrix used in pqEDMD that maps from the function space to the original output space. Furthermore, we assume that system~\eqref{eq:lifted_subspace} is minimal, in the sense that the observability matrix of the pair $(A_S,\,C_S)$ is of full row rank, and the controllability matrix of the pair $(A_S, B_S)$ is of full columns rank. We propose that this translates to the original non-linear system~\eqref{eq:nonlin_affine_system} being both observable and reachable.

The solution to the lifted dynamics~\eqref{eq:lifted_subspace} corresponds to the successive application of $A_S$ on the lifted initial state, along with the effect of the input sequence:

\begin{equation}
	\Phi(x(k)) = A_S^k \Phi(x(0)) + \sum_{i=0}^{k-1}A_S^{k-i-1}B_S u(i).
	\label{eq:lift_DS_solution}
\end{equation}

For each available sample, the input-output relationship of the system in the function space is derived from the system solution~\eqref{eq:lift_DS_solution}, the lifted output equation~\eqref{eq:ls_psi_output}, and the lifted Hankel matrices for the output~\eqref{eq:lifted_hankel_output}, along with their input counterparts using $b_H$ block rows. We further define the extended lifted-observability matrix $\Gamma_{b_H}\in\mathbb{R}^{b_Hd\times d_s}$ built the same way as~\eqref{eq:Extended_observability_mat} and the lifted-Markov parameter matrix $H_{b_H} \in \mathbb{R}^{b_Hd \times b_H m}$ built the same way as~\eqref{eq:markov_parameter}, both constructed with respect to $A_S$, $B_S$ and $C_S$. These components together give the input-output relationship,

\begin{equation}
	\left\{\Psi^{i}_{0|b_H-1} = \Gamma_{b_H}\Xi_{0\colon T_i-b_H} + H_{b_H}U_{0|b_H-1}\right\}_{i=1}^{E},
	\label{eq:lift_system_obs_mark}
\end{equation}
where $\{\Xi_{0\colon T_i-b_H}\}_{i=1}^{E}$ is a set of sequences of lifted state evaluations,
\begin{equation}
	\left\{\Xi_{0\colon T_i-b_H} =
		\begin{bmatrix}
			\Phi(x(0)) &  \cdots & \Phi(x(T_i-b_H))
	\end{bmatrix}\right\}_{i=1}^{E}.
	\label{eq:lift_state_sequence}
\end{equation}
\subsection{Lifted SVD}\label{sub:lifted SVD} 
This section shows the calculation of the singular value decomposition of the $W_1\mathcal{O}W_2$ matrices, that come from the selected subspace algorithm, and the different projections~(\ref{eq:Oi_mat},\ref{eq:Oi_stochastic}). The SVD of the $W_1\mathcal{O}W_2$ matrix produces a lifted-observability matrix $\Gamma_{b_H}$. The lifted-observability matrix, along with the $\mathcal{O}$ matrix, give as a result a sequence of lifted states~\eqref{eq:lift_state_sequence}. Finally, from the sequence of lifted states we can get the operator $A_S$, and the system matrices $B_S$ and $C_S$.
\subsubsection{Autonomous deterministic system}\label{sec:The Unforced-Deterministic System} 
The main idea for an autonomous and deterministic system is to use the data equation~\eqref{eq:lift_system_obs_mark} to estimate the column space of the extended observability matrix $\Gamma_{b_H}$, in the same fashion as the estimation of the extended observability matrix for linear systems~\eqref{eq:lin_sys}. For unforced-deterministic systems in a function space, the data equation~\eqref{eq:lift_system_obs_mark} becomes,
\begin{equation}
	\left\{\Psi^{i}_{0|b_H-1} = \Gamma_{b_H}\Xi_{0|T_i-b_H}\right\}_{i=1}^{E},
	\label{eq:set_unf_data_eq}
\end{equation}
which is equivalent to the data equation~\eqref{eq:autonomous_system_obs} for linear systems. The main addition to the analysis is that equation~\eqref{eq:set_unf_data_eq} is valid for all the available sample trajectories of the system. Then, the concatenation of the output Hankel matrices and the sequences of states satisfies the following data equation,
\begin{equation}
	\Psi^Y = \Gamma_{b_H}\Xi
	\label{eq:unf_data_eq}
\end{equation}
where the $\Psi^Y$ and $\Xi$ matrices, are the concatenation of all elements in their respective sets,
\begin{align}
	\Psi^Y & =
	\begin{bmatrix}
		\Psi^{1}_{0|b_H-1} & \cdots & \Psi^{E}_{0|b_H-1}
	\end{bmatrix}\label{eq:lift_out_concat} \\
	\Xi    & =
	\begin{bmatrix}
		\Xi_{0|j_1-1} & \cdots & \Xi_{0|j_E-1}
	\end{bmatrix}.\label{eq:state_seq_concat}
\end{align}
It is clear from~\eqref{eq:unf_data_eq} that the Hankel sequences of lifted outputs is a linear combination of the extended observability matrix. Similar to the linear case~\eqref{eq:autonomous_system_obs} it is necessary to have a set of initial conditions sufficiently informative to perform the identification. Then, the matrix $\mathcal{O}_{b_H}$, now dependent on the total number of Hankel blocks for the \textit{past} and \textit{future} is,
\begin{equation}
	\mathcal{O}_{b_H} = \Psi^Y,
	\label{eq:O_bH_det_unforced}
\end{equation}
and the SVD that provides the column space of the extended observability matrix is,
\begin{equation}
	W_1\mathcal{O}_{b_H}W_2 =  \mathcal{U}_{d_s}\Sigma_{d_s}\mathcal{V}_{d_s}^T,
	\label{eq:WOW_SVD_unforced}
\end{equation}
where the $W$ matrices are the identity, and the truncation of the SVD depends on the order of the lifted sequence of states $d_s$. In the later, we will give an analysis regarding the choice of this parameter.
\subsubsection{Stochastic system}\label{sec:The stochastic system} 
For the lifted stochastic system, we use the set of lifted block Hankel matrices in the \textit{future} $\{\Psi^{i}_f=\Psi^{i}_{b_p|b_H-1}\}_{i=1}^{E}$, the set block Hankel matrices of future inputs $\{U_f^i\}_{i=1}^{E}$, an extended observability matrix with appropriate dimensions $\Gamma_{f}\in\mathbb{R}^{b_fd\times d_s}$ and a Markov parameter matrix $H_{f}\in\mathbb{R}^{b_fd\times b_fm}$ to get the concatenated data equation,
\begin{equation}
	\Psi^Y_{f} = \Gamma_{f}\Xi_f + H_fU_f,
	\label{eq:fut_lift_data_eq}
\end{equation}
Similar to the linear system with inputs~\eqref{eq:Oi_stochastic}, the definition of $\mathcal{O}_{b_f}$ for the lifted subspace identification uses the instrumental variables and an oblique projection,
\begin{equation}
	\mathcal{O}_{b_f} = \Psi^Y_{f}\Big/_{U_f}\mathbf{W}_{\Psi_p},
	\label{eq:O_bf_stochastic}
\end{equation}
where the matrix $W_{\Psi_p}$ is the horizontal concatenation of the individual instrumental variables from all the available experiments of the system,
\begin{equation}
	W_{\Psi_p} =
	\begin{bmatrix}
		W^1_{\Psi_p} & \cdots & W^E_{\Psi_p}
	\end{bmatrix}.
	\label{eq:inst_concat}
\end{equation}

Applying the weighting matrices that depend on the chosen algorithm, the SVD that provides an approximation of the extended observability matrix is,
\begin{equation}
	W_1\mathcal{O}_{b_f}W_2 = \mathcal{U}_{d_s}\Sigma_{d_s}\mathcal{V}_{d_s}.
	\label{eq:WOW_SVD_stoch}
\end{equation}
Then, an approximation of the extended observability matrix for the stochastic system in the lifted function space is,
\begin{equation}
	\Gamma_f = W_1^{-1}\mathcal{U}_{d_s}\Sigma_{d_s}^{1/2}.
	\label{eq:lift_ext_obs}
\end{equation}

Notice that if the system is autonomous and stochastic, the oblique projection~\eqref{eq:O_bf_stochastic} reduces to the orthogonal projection of $\Psi^Y_{f}$ onto the row space of the matrix of instrumental variables $W_{\Psi_p}$, the $W_{\{1,2\}}$ weighting matrices are the identity, and the calculation of the extended observability matrix for the function space remains the same.

With the extended observability matrix, the next step is to set the order of the lifted sequence of states $d_s$ and derive the operator $A_S$ and the remaining system matrices $B_S$ and $D_S$.
\subsection{Lifted system matrices}\label{sub:The lifted system matrices} 
The method to get the lifted systems matrices is analogous to the linear case from equations~(\eqref{eq:state_seq}-\eqref{eq:matrices_least_squares}), where the objective is to get a sequence of lifted states from the extended observability matrix and set up a least squares problem. There are methods in the literature that provide the $A$ and $C$ matrices directly form the analysis of shift structure of the extended observability matrix~\cite{overschee96}. The remaining matrices come from the knowledge of $A$ and $C$, and the analysis of the data equation~(\ref{eq:input_output_matrix},\ref{eq:fut_lift_data_eq}). Instead of that, we will adapt the combined algorithm 2 from~\cite{overschee96}.
\subsubsection{Order of the lifted subspace}\label{sec:The order of lifted subspace} 
For the subspace identification of linear systems, the order $n$ of the systems corresponds to the number of nonzero singular values from the SVD~\eqref{eq:linear_svd}, and if the system is stochastic, the order comes from the ability to distinguish the $n$ disturbed singular values from the remaining $il-n$; that are non-zero because of the noise. The rule to distinguish them, depends on the gap between the $n^{\text{th}}$ and the $(n+1)^{\text{th}}$ singular value~\cite{verhaegen_verdult_2007}. We propose to select the singular values grater than an arbitrary precision.

From the assumption that the system is minimal, reachable (controllable), and observable, the number of outputs must be greater than or equal to the number of states. The output of the system in the function space is the lifted output $y(t)$ of the system via the vector of observables $\Psi$, that has cardinality $d$. Then, the order of the lifted state must satisfy,
\begin{equation}
	d\leq d_s\leq \min\left\{r\colon \sigma_r\leq\epsilon\sigma_1\sum_{i=1}^E T_i\right\}.
	\label{eq:n_s_cond}
\end{equation}
Even though this condition does not solve the problem of selecting an order to the system, or provide a clear rule regarding the choice of $p$, $q$, \textit{future} and \textit{past} blocks parameters, it does provide a bound to the choice of the order of the lifted subspace.

\subsubsection{Lifted state sequence and solution}\label{sec:Lifted state sequence and solution} 
The set of sequences of lifted states $\{\Xi_{0\colon T_i-b_H}\}_{i=1}^{E}$ for the formulation of the least squares problem will not follow the formulation for linear systems from equations~(\ref{eq:state_seq}-\ref{eq:matrices_least_squares}) because it needs to be adapted for the case of multiple experiments. In addition, empirical testing of equations~(\ref{eq:shift_oblique_projection} and \ref{eq:shifted_future_states}) for the calculation of the shifted and lifted states for the formulation of the least squares solution does not produce accurate approximations. Instead, the formulation of the least squares problem comes from the calculation of independent lifted state sequences for each of the available experiments, and performing a division of sequences into two distinct ones, the \textit{past} and \textit{future} values, the sames as the division of the data for the pqEDMD from~\eqref{eq:edmd_data_mat_i}.

Applying~\eqref{eq:state_seq} to each of the available experiments, yields the following set of approximations of the lifted state sequences,
\begin{equation}
	\{\hat{\Xi}_{0|T_i-b_H} = \Gamma_{b_f}^{\dagger}\mathcal{O}^i_{b_f}\}_{i=1}^E,
	\label{eq:lifted_state_se}
\end{equation}
and depending if the system, i.e., deterministic, or not, autonomous or not, the set of matrices $\{\mathcal{O}_{b_f}^i\}_{i=1}^E$ is the application of~(\ref{eq:O_bH_det_unforced} or \ref{eq:O_bf_stochastic}).

From the approximation of the lifted sequences of states, define the \textit{past} and \textit{future} lifted state sequences and the \textit{past} matrices of inputs and lifted output as,
\begin{subequations}
	\label{eq:lift_edmd_mat_i}
	\begin{align}
		\bigl\{\hat{\Xi}^i_+\bigr. & = \bigl.
			\begin{bmatrix}
				\hat{\Phi}^i(x(1)) & \cdots & \hat{\Phi}^i(x(T_i-b_f))
		\end{bmatrix}\bigr\}_{i=1}^E     \\
		\bigl\{\hat{\Xi}^i_-\bigr. & = \left.
			\begin{bmatrix}
				\hat{\Phi}^i(x(0)) & \cdots & \hat{\xi}^i(T_i-b_f-1)
		\end{bmatrix}\right\}_{i=1}^E,     \\
		\Bigl\{U^i_-\bigr.         & = \bigl.
			\begin{bmatrix}
				u^i_{b_p|b_p}(0) & \cdots & u^i_{b_p|b_p}(T_i-b_f-1)
		\end{bmatrix}\Bigr\}_{i=1}^E \\
		\Bigl\{\Psi^{i}_-\Bigr.  & = \Bigl.
			\begin{bmatrix}
				\Psi^{i}_{b_p|b_p}(0) & \cdots & \Psi^{i}_{b_p|b_p}(T_i-b_f-1)
		\end{bmatrix}\Bigr\}_{i=1}^E
	\end{align}
\end{subequations}
where each $U^i_-$ is the first block of the \textit{future} Hankel matrix of inputs sliced from the first element until the penultimate. Then, we concatenate them in the \textit{past} and \textit{future} matrices for the formulation of the lest squares problem,
\begin{subequations}
	\label{eq:lift_edmd_mat}
	\begin{align}
		\hat{\Xi}_+ & =
		\begin{bmatrix}
			\hat{\Xi}^1_+ & \cdots & \hat{\Xi}^E_+
		\end{bmatrix} \\
		\hat{\Xi}_- & =
		\begin{bmatrix}
			\hat{\Xi}^1_- & \cdots & \hat{\Xi}^E_-
		\end{bmatrix} \\
		U_-         & =
		\begin{bmatrix}
			U^1_- & \cdots & U^E_-
		\end{bmatrix}                 \\
		\Psi^{Y}_-  & =
		\begin{bmatrix}
			\Psi^{Y_1}_- & \cdots & \Psi^{Y_E}_-
		\end{bmatrix}
	\end{align}
\end{subequations}

With the two sequences of lifted states, and the first block of the data and input Hankel matrices properly sliced, the least squares problem for the calculation of the $A_S$ operator, and the system matrices $B_S$ and $C_S$ is,
\begin{align}
	\hat{\Xi}_+ & =
	\begin{bmatrix}
		A_S & B_S
	\end{bmatrix}                                 %
	\begin{bmatrix}
		\hat{\Xi}_- \\U_-
	\end{bmatrix} + r_{[A_S\,B_S]}\label{eq:AB_lift_lsq}           \\
	\Psi^Y_-    & = C_S\hat{\Xi}_+ + r_{C_S},\label{eq:C_lift_lsq}
\end{align}
where $r_{[A_S\,B_S]}$ and $r_{C_S}$ are the residual term to minimize. Notice that in~\eqref{eq:C_lift_lsq}, $C_S$ acts on the sequence of states of the \textit{future}, to produce the lifted output of the \textit{past}. This is an empirical result from the implementation of the algorithm, the solution with the two slicings in the past or in the future increases the identification error.
\subsection{Prediction}\label{sec:Prediction} 
With the solution of the least squares problems~(\ref{eq:AB_lift_lsq}, \ref{eq:C_lift_lsq} and \ref{eq:C_least_sq}), the linear evolution of the lifted space and its relationship with the evolution of the measurements of the system~\eqref{eq:lifted_subspace} is complete. Then, for the prediction of the system dynamics from an arbitrary initial condition, it is necessary to have a relationship between an initial state of the output, and the initial state of the lifted subspace. Assuming that the initial state of the measurements $y(0)$ is the only available information for the prediction of all the subsequent outputs, the initial lifted state $\xi(0)$ is the lifted output, weighted by the Moore-Penrose pseudo inverse of the lifted subspace matrix $C_S$,
\begin{equation}
	\xi(k) = C_S^{\dagger}\Psi(y(k)),
	\label{eq:to_lift_spacwe}
\end{equation}
with the caveat of imposing the cardinality of the subspace to be equal to the observables, $d_S=d$. Also, notice that this applies for any arbitrary time instant $k$. This is an important consideration, because in most cases, the vanilla EDMD, and this subspace extension, do not span an invariant subspace of the dynamical system in a functions space. That is, they do not span the Koopman operator. In discrete time, the Koopman operator of an unforced nonlinear dynamical system $x(k+1)=f(x(k))$ satisfies,
\[
	\mathcal{K}\phi = \phi\circ f,
\]
where the set of functions $\phi:\mathbb{R}^n \rightarrow \mathbb{R}$ are the observables, and $\mathcal{K}$ is the infinite-dimensional linear operator that relates the evolution of observables to the evolution of states. In the limit of data and observables, the EDMD based algorithms span the invariant Koopman subspace~\cite{Williams2015} and,
\[
	\Psi(y(k)) = A_D^k\Psi(y(0)).
\]
Evolving an arbitrary initial condition $y(0)$, $k$ time-steps ahead is equivalent to applying the operator $A_D$, $k$ times and then bringing it back to the output space. This in general, is not the case (none of the examples in this paper span a Koopman invariant subspace, despite the accurate approximations). Then, the linear approximation of the output of the system via a finite set of observable functions is only valid for a finite number of steps-ahead, and therefore, to have the best approximation of the dynamics of the system, it is necessary to perform a correction for every step in the evolution of a trajectory~\cite{Junker22}. Then, for the pqSEDMD, we use the same strategy for the prediction of the output from an arbitrary initial output: 1) lift the initial output with $\Psi$. 2) transform into the lifted state with~\eqref{eq:to_lift_spacwe}. 3) perform the evolution of the lifted state one step ahead with~\eqref{eq:ls_dynamics}. 4) transform the lifted state back into the output space with~\eqref{eq:ls_psi_output}. 5) calculate the evolved output with~\eqref{eq:ls_y_output}. 6) repeat from step 1 for the necessary number of steps $k$.

Regarding the transformation of the evaluation of the observables via the Moore-Penrose pseudo inverse of the lifted state to lifted output we can also transform the whole system in the lifted subspace into a regular EDMD structure~\eqref{eq:EDMD} via a transformation of the original set of observables with the $C_S^{\dagger}$ matrix.
\subsection{Lifted rank conditions and persistence of excitation}\label{sec:Rank conditions and persistence of excitation} 
For the case of deterministic lifted-subspace identification of autonomous systems, the data equation~\eqref{eq:unf_data_eq} states that the lifted Hankel matrix of outputs are a linear combination of of the columns of the extended observability matrix $\Gamma_{b_H}$. For the approximation of the system in the lifted function space, we must also have a set of initial conditions whose trajectories provide enough information to reveal the nonlinear dynamics. To guarantee that the columns spaces are equal, i.e., $\texttt{col}(\Psi^Y)=\texttt{col}(\Gamma_{b_H})$, the following conditions must be met: (i) The number of Hankel Blocks is greater than the dimension of the function space of the state, i.e., $b_{b_H}>d_s$. (ii) The number of of columns in each of the Hankel matrix of outputs is greater then the number of blocks, i.e., $\{j_i>H_{b_H}\}_i=0^E$. Finally, (iii) The sequences of lifted states $\{\Xi^i\}_{i=1}^E$ has full row rank $d_s$.

\begin{lemma}[Lifted subspace for autonomous systems]
Consider an autonomous system~\eqref{eq:nonlin_affine_system},  and assume that there exists an autonomous lifted subspace approximation~\eqref{eq:lifted_subspace}. If the system in the lifted subspace is observable, if the initial conditions are informative enough, and if the resulting sequences of lifted output are long enough to have $b_H>d_s$ Hankel blocks and $\{j_i>d_s\}_{i=1}^E$ columns, then
\begin{equation}
\texttt{rank}(\Psi^Y)=d_s
\label{eq:lift_hankel_out_rank_cond}
\end{equation}
\end{lemma}
\begin{proof}
From the assumption that the system is observable, $\texttt{rank}(\Gamma_{b_H})=d_s$. All the sequences of lifted state evaluations $\{\Xi_{0\colon T_i-b_H}\}_{i=1}^{E}$ are,
\begin{equation}
\left\{
\begin{bmatrix}
\Phi(x(0)) & \cdots & A_S^{T_i-b_H}\Phi(x(T_i-b_H))
\end{bmatrix}\right\}_{i=1}^{E},
\label{eq:A_lift_state_squence}
\end{equation}
and $\Xi$ is the horizontal concatenation of this sequences. From the assumption that the initial conditions are informative enough, $\texttt{rank}(\Xi) = d_s$. Applying Sylvester's inequality to~\eqref{eq:unf_data_eq} implies that $\texttt{rank}(\Psi^Y)=d_s$.
\end{proof}
\begin{remark}
The observability assumption is related to the lifted subspace approximation. Meaning that there exist a full rank $(A_s,\,C_s)$ pair.
\end{remark}

The previous condition on the state sequence is necessary but not sufficient, we can show this by a counter example using the autonomous Duffing equation. Figure~\ref{fig:aut_duff} shows two approximations of the trajectories of the Duffing equation where the training set comes from two different initial conditions.
\begin{figure}[ht]
\begin{center}
\includegraphics[width=0.95\linewidth]{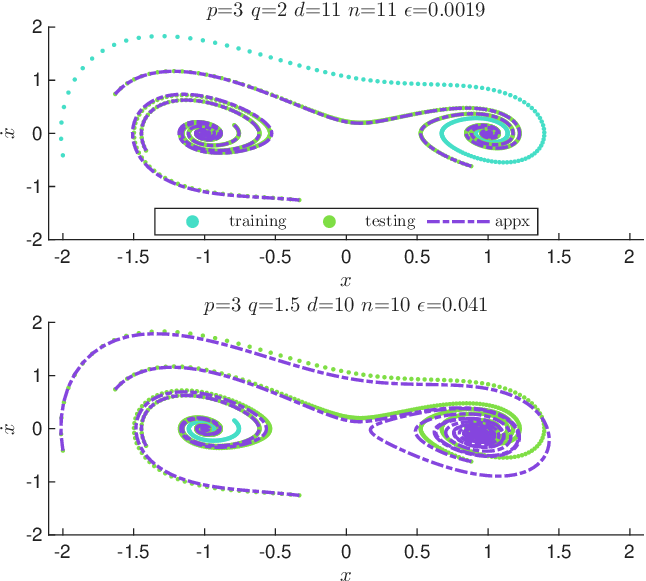}
\end{center}
\caption{Performance of the pqSEDMD algorithm for two different sets of training data, i.e., different initial conditions, while keeping everything else equal. While $\texttt{rank}(\hat{\Xi}_{0|j_{\text{top}}-1})=\texttt{rank}(\hat{\Xi}_{0|j_{\text{bot}}-1})=27$, the bottom case does not provide sufficient information.}\label{fig:aut_duff}
\end{figure}

The two approximations have the same parameters, and satisfy the rank condition on the sequence of lifted states. Nevertheless, the bottom approximation, does not provide sufficient information for the identification of the nonlinear dynamics of the system. Similar to the different flavors of the EDMD algorithm and Koopman operator approximations, the necessary amount of data, the initial conditions and sampling time that provides an accurate approximation depend on the dynamical systems to approximate. For the particular case of the Duffing oscillator in Figure~\ref{fig:aut_duff}, taking the outermost trajectory of the system provides sufficient information to model the two asymptotically stable attractors, from one single trajectory.

In summary, if there are enough points in the dataset and those points come from one or many informative enough initial conditions, then the concatenated sequence of lifted states will have full $\texttt{rank}(\Xi)=d_s$ making the lifted-observability matrix identifiable.

For the general case of non autonomous systems, we will also assume that there are many available trajectories of the system, and propose some rank conditions on the lifted state sequence and Hankel matrix of inputs as sufficient conditions to identify the nonlinear dynamics.

For a concatenated set of lifted state sequences $\Xi$~\eqref{eq:state_seq_concat} and the corresponding concatenation of the set of Hankel matrices of inputs $U$, the rank condition to satisfy is:

\begin{lemma}[Lifted rank condition]
\label{lem:lifted_state_inpu_sequence}
Consider a controllable lifted system~\eqref{eq:lifted_subspace} and a set of input-to-lifted state trajectories,
{\small\[
\left\{{(\Phi,u)}^{i}=\{\Phi(x(k))^{i};u(k)^i\}_{k=0}^{T_i-1}\in\mathbb{R}^{(d_s+m)\times T_i}\right\}_{i=1}^E,
\]}
from zero initial conditions. If the horizontal concatenation of the set of inputs $u^i$ is persistently exciting of order $d_s+b_{H}m$, then
\[
\texttt{rank}\left(
\begin{bmatrix}\Xi \\U
\end{bmatrix}\right) = d_s + b_{H}m
\]
then,
\[
\texttt{rank}\left(\Psi^Y\Pi_{U^\bot}\right) = d_s
\]
and,
\[
\texttt{col}\left(\Psi^Y\Pi_{U^\bot}\right) = \texttt{col}\left(\Gamma_{b_H}\right)
\]
\end{lemma}
\begin{proof}
The matrix has $d_s + b_{H}m$ rows, then it is sufficient to show that all its rows are linearly independent. Assuming the opposite means that there exists a nonzero lifted state vector $\eta\in\mathbb{R}^{d_s}$ and a nonzero input vector $\omega\in\mathbb{R}^{b_{H}m}$ such that,
\[
\eta^T\Xi + \omega^T U = 0.
\]
Consider a set of sequences of lifted state evaluations $\{\Xi_{0}^i=\Xi_{0\colon T_i - b_{H}}\}_{i=1}^{E}$ as defined by~\eqref{eq:lift_state_sequence}. According to~\eqref{eq:lift_DS_solution}, the evolution of those sequences $b_p$ steps ahead is $\Xi_{b_p}^{i} = A_S^{b_p}\Xi_{0}^i + \mathcal{C}_{b_p}U_{p}^i$, where the controllability matrix $\mathcal{C}_{b_p}=[B_S\,A_SB_S\,\cdots\,A_S^{b_p-1}B_S]$, and $U_p^i$ are the Hankel matrices of input sequences that drive the lifted states. If the linear dependence holds for the concatenation of the sequences of states, then it also holds for the future sequences of states, $\eta^T\Xi_{b_p}+\omega^TU_f=0$. Substituting the evolution of the initial sequences and rearranging the linear dependence yields,
\begin{align*}
\eta^T(A_S^{b_p}\Xi_0+\mathcal{C}_{b_p}U_p)+\omega^TU_f&=0\\
\eta^TA_S\Xi_0 +
\begin{bmatrix}\eta^T\mathcal{C}_{b_p}&\omega^T
\end{bmatrix}
\begin{bmatrix}U_p\\U_f
\end{bmatrix}&=0.
\end{align*}
Post-multiplying by a matrix $\Pi_{\Xi_{0}^\bot}$ that satisfies $\Xi_0\Pi_{\Xi_{0}^\bot}=0$, gives the linear dependent relation $[\eta^T\mathcal{C}_{b_p}\;\omega^T]U\Pi_{\Xi_0^\bot}=0$. From the assumption that the sequence of inputs is persistently exciting, the Hankel matrix of inputs $U$ has full row rank, even projected onto the orthogonal complement of the row space of lifted states. Implying that $\omega^T=0$, and $\eta^T\mathcal{C}_{b_p}=0$. For the latter, the controllability assumption implies that $\texttt{rank}(\mathcal{C}_{b_p})=b_p$, and that $\eta^T=0$ arriving at a contradiction.
\end{proof}
The consequence of lemma~\ref{lem:lifted_state_inpu_sequence} is the identifiability of the lifted system. Nonetheless, the rank condition is sufficient but not necessary, we can show this with a particular case. The top of figure~\ref{fig:step_duff} shows the identification of the Duffing oscillator where the forcing signal is a step response, and the training set is a single trajectory whose initial condition is the furthest from the origin. Even though the rank of the Hankel matrix of inputs is 1, the pqSEDMD is able to capture the dynamics of the two attractors.
\begin{figure}[ht]
\begin{center}
\includegraphics[width=0.95\linewidth]{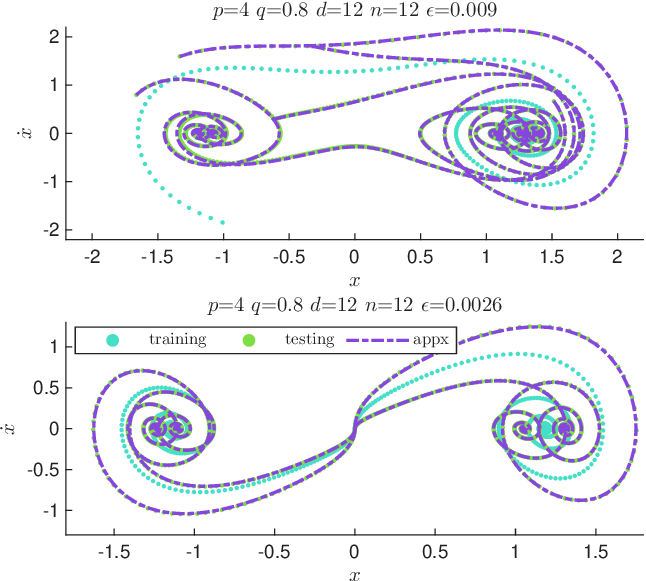}
\end{center}
\caption{Performance of the pqSEDMD algorithm for two different sets of experiments of the Duffing oscillator with a constant input where $\texttt{rank}(U)=1$.}\label{fig:step_duff}
\end{figure}

The bottom of figure~\ref{fig:step_duff} shows the result of performing the identification of the step response with experimental data whose initial conditions are the origin, and the forcing signal is still a constant function. This setup challenges the hypothesis that the nonzero initial conditions of the training set is what provides the necessary data for the identification. In this case, it is not enough to have a single trajectory to one of the attractors to capture the nonlinear dynamics of both. Instead, we can identify with one experiment per attractor, without having to select the outermost trajectory. Having a Hankel matrix of inputs that does not have a full row rank demands a modification of the solution regarding the projections.
\subsection{Projections of a lifted subspace}\label{sub:Projections of a lifted subspace} 
The pqSEDMD relies in the oblique projection~\eqref{eq:O_bf_stochastic} that returns the product of the Hankel matrix of lifted outputs and $\Pi_{U^\bot}$ defined in~\eqref{eq:pi_u_orthogonal}. The whole process is equivalent to having a geometric operator that projects the row space of the Hankel matrix of \textit{future}-lifted outputs onto the row space of the instrumental variables, along the row space of the Hankel matrix of future inputs. It is common to find in the literature~\cite{overschee96,Verhaegen1992,verhaegen_verdult_2007} that an RQ decomposition is a numerically efficient method for calculating the orthogonal and oblique projections necessary for the estimation of the extended observability matrix.

The restriction with the RQ decomposition and the efficient calculation of the projection is the necessity to work with matrices that are not rank deficient~\cite{Golub_2013}. Implementing those methods for the pqSEDMD algorithm and the rank deficient Hankel matrices of inputs produces numerically inaccurate results compared to the use of the formula for the projection relaxing the inverse with a Moore-Penrose pseudo inverse. Recall that we perform the projection because post multiplying $\Psi^{Y}_f$ with $\Pi_{U^\bot}$ eliminates the effect of the input from~\eqref{eq:fut_lift_data_eq}. Following~\cite{Viberg95}, the deduction of $\Pi_{U^\bot}$ comes from an approximation of the Markov parameter matrix $H_f$ via the least squares problem, $\min_{\hat{H}_f}{\Vert \Psi^Y_f-\hat{H}_fU_f \Vert}^{2}_{\text{F}}$. When the Hankel matrix is well conditioned, the solution is $\hat{H}_f=(\Psi^Y_f U_f^T)(U_f U_f^T)^{-1}$. Substituting the solution into the original least squares problem,
\begin{subequations}
\label{eq:lift_projection}
\begin{align}
\Psi^Y_f-\hat{H}_fU_f & = \Psi^Y\left(I - U_f^T\left(U_fU_f^T\right)^{-1}U_f\right)\label{eq:lift_projection_formula} \\
& = \Psi^Y\Big/\mathbf{U}_f^{\bot},
\end{align}
\end{subequations}
a result equivalent to formula~\eqref{eq:Oi_mat}, stating the necessity of a full row-rank Hankel matrix of inputs. We argue the use of the SVD of $U_f$ along with its effective rank to eliminate the effect of the input, relaxes the necessity to have a full (row) rank Hankel matrix of inputs. For example, consider the identification of the duffing oscillator from six trajectories starting at the origin excited with a step input: two non-ideal conditions for the approximation of the system dynamics. And instead of testing with the same settings, consider a test set with non-zero initial conditions and a cosine forcing signal. The result is an accurate identification of the dynamics as depicted in figure~\ref{fig:step_cos_duff}.
\begin{figure}[ht]
\begin{center}
\includegraphics[width=0.95\linewidth]{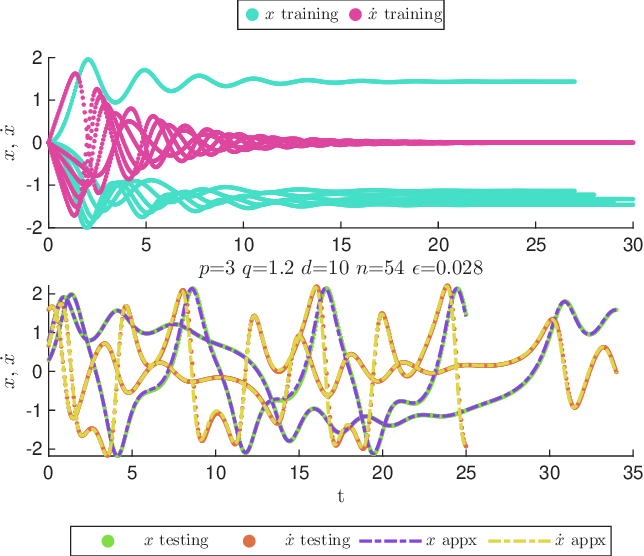}
\end{center}
\caption{Performance of the pqSEDMD algorithm trained with six trajectories from the origin excited with a step response, and tested with two trajectories from random initial conditions, forced with cosine functions.}\label{fig:step_cos_duff}
\end{figure}

\section{Results}\label{sec:Results} 
The \textit{Temperature Control Lab} (TCLab) is a compact and open-source platform from APMonitor\textsuperscript{\textregistered} designed for educational and research purposes in control and system identification. It serves as an excellent real-world test-bench for testing the performance of advanced identification and control algorithms due to its inherent nonlinear dynamics, the thermal delays and its susceptibility to environmental disturbances, like temperature and electromagnetic noise, and quantization errors.

The platform consists of two heaters (Q1 and Q2) with two corresponding temperature sensors (T1 and T2). The heaters are the manipulated variables of the system: the heat output is adjusted by modulating the current flow to each of the two BJT transistors, and the output is the temperature measurement via two thermistors. The dynamics of the system are the transmission of thermal energy to the sensors by conduction, convection, and radiation.

The data for the implementation of the algorithm comes from three hour long experiment and an additional sixteen-minutes long experiment of the platform using the two heaters. From these samples, the three hours of data serve as the training set and the last experiment for testing. The heat input sequences are two pseudo-aleatory signals at different power values between 0\% and 90\% that change every 4 or 6 minutes. The sampling time is set to 2 seconds, providing 5400 total data-points for training the approximation of the TCLab dynamics, and 460 for the validation. Table~\ref{tab:heat_input_parameters} shows a summary of the simulation parameters.

\begin{table}
\caption{Experimental parameters}\label{tab:heat_input_parameters}
\begin{center}
\begin{tabular}[c]{lrr}
\hline
\textbf{Parameter} & \textbf{Value}     & \textbf{Description}     \\
\hline
Duration           & $[3600\,,960]$ [s] & Total exp time             \\
$\Delta t$         & 2 [s]              & Sampling time            \\
$\Delta $ Q1       & $[240\,,360]$ [s]  & 4 or 6 [min] \\
$\Delta $ Q2       & $[240\,,360]$ [s]  & 4 or 6 [min] \\
$N$                & 5400 [-]           & Training points          \\
$N_{ts}$           & 480 [-]            & Testing points           \\

\hline
\end{tabular}
\end{center}
\end{table}

Figure~\ref{fig:tcl_results} shows the results of the approximation, where there are 7 observables based on Legendre orthogonal polynomials, a function subspace of order 7. This gives a \textit{mean relative absolute error} (MRAE) of 0.044. The approximation of the dynamics starts from the known initial temperature of the two heaters, propagated through the 16 minutes of available heat input data from the heaters.

\begin{figure}[ht]
\begin{center}
\includegraphics[width=0.95\linewidth]{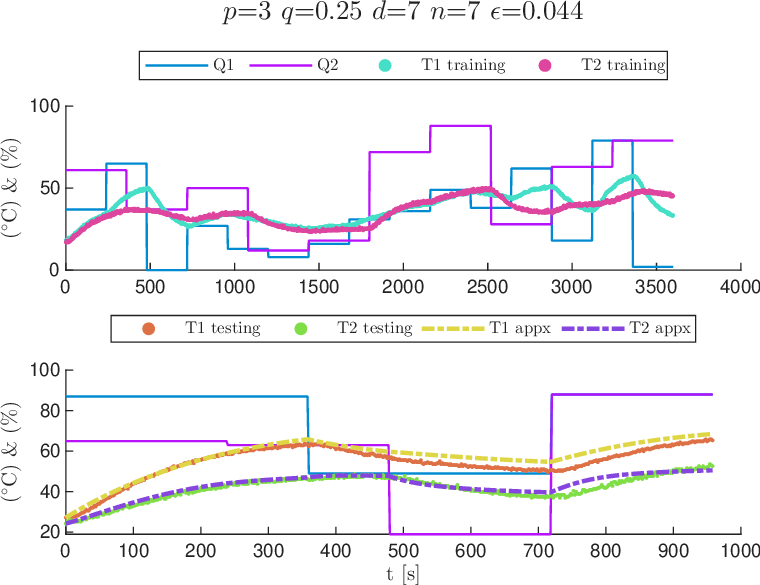}
\end{center}
\caption{Performance of the pqSEDMD algorithm. Trained from 3 trajectories excited with a sequence of step signals for the two heaters. The top figure shows an example of one of the training trajectories,and the bottom, shows the testing trajectory.}\label{fig:tcl_results}
\end{figure}

\section{Conclusions}\label{sec:Conclusions} 
This paper introduces the pqSEDMD, an algorithm that combines the complementary strengths of two identification paradigms. Adapting the subspace identification methods to operate in a lifted function space achieves a robust approximation of nonlinear dynamics from noisy and limited data. Subspace methods are robust, able to handle process and measurement noise through Hankel matrix projections for the identification of linear dynamics; EDMD methods are broader, able to handle nonlinear systems through the use of observable functions for the identification of systems with low uncertainty. The pqSEDMD combines these advantages in a computationally stable algorithm for data-driven modeling of nonlinear dynamical systems.

The paper also provides an analysis of the sufficient rank conditions of the Hankel matrix of lifted outputs, the Hankel matrix of inputs, and the initial conditions for the identifiability of the system. The approach is correlate to the original formulation of the fundamental lemma by Willems~\cite{Willems04}, and provides some bounds on the order of the lifted function space.

The simulation results show the effectiveness of the algorithm for both deterministic and stochastic systems. The  Duffing oscillator shows that the algorithm able to generalize and identify unseen dynamics, even if the convergene of the trajectories is to a different attractor. Additionally, the tests on the TCLab show the effectiveness of the algorithm on real experimental data.
\bibliographystyle{ieeetr}
\bibliography{siddec}
\end{document}